\documentclass[11pt]{amsart}
\usepackage[margin=1.5in]{geometry} 

\usepackage{amssymb,enumitem}
\usepackage{amsmath}
\usepackage{amsthm}
\usepackage{mathrsfs}
\usepackage{hyperref}
\usepackage{amsrefs}
\usepackage{comment}
\usepackage{amsthm,amsrefs,amssymb,amsmath}
\numberwithin{equation}{section}

\newtheorem{thm}{Theorem}[section]
\newtheorem{prop}[thm]{Proposition}
\newtheorem{lemma}[thm]{Lemma}

\newtheorem{cor}[thm]{Corollary}
\theoremstyle{definition}
\newtheorem{definition}[thm]{Definition}

\newtheorem*{example}{Example}
\newtheorem*{remark}{Remark}

\newcommand{\Z}{\mathbb{Z}}
\newcommand{\N}{\mathbb{N}}
\newcommand{\rt}{\sqrt}
\newcommand{\ra}{\rangle}
\newcommand{\la}{\langle}
\newcommand{\bd}{\begin{description}}
\newcommand{\ed}{\end{description}}
\newcommand{\ds}{\displaystyle}
\newcommand{\ord}{\operatorname{ord}}
\newcommand{\lcm}{\operatorname{lcm}}
\newcommand{\Sym}{\operatorname{Sym}}
\newcommand{\Alt}{\operatorname{Alt}}
\newcommand{\SL}{\operatorname{SL}}

\title[Asymptotics and congruences for partition functions]{Asymptotics and congruences for partition functions which arise from finitary permutation groups}
\date{\today}
\author{Tessa Cotron}
\address{605 Asbury Circle, Box 122042, Atlanta, GA 30322}
\email{tessa.cotron@emory.edu}

\author{Robert Dicks}
\address{605 Asbury Circle, Atlanta, GA 30322}
\email{rdicks@emory.edu}

\author{Sarah Fleming}
\address{1192 Paresky Center, Williams College, Williamstown, MA 01267}
\email{smf1@williams.edu}

\begin{document}
\begin{abstract}
In a recent paper, Bacher and de la Harpe study the conjugacy growth series of finitary permutation groups.  In the course of studying the coefficients of a series related to the finitary alternating group, they introduce generalized partition functions $p(n)_{\textbf{e}}$.  The group theory motivates the study of the asymptotics for these functions.  Moreover, Bacher and de la Harpe also conjecture over 200 congruences for these functions which are analogous to the Ramanujan congruences for the unrestricted partition function $p(n)$.  We obtain asymptotic formulas for all of the $p(n)_{\textbf{e}}$ and prove their conjectured congruences.
\end{abstract}
\maketitle

\section{Introduction}

In~\cite{bacher}, Bacher and de la Harpe study infinite permutation groups that are locally finite. If $X$ is a nonempty set, given a permutation $g$ of $X$, the \textit{support} of $g$ is $\sup(g)=\{x\in X : g(x)\neq x\}$. The \textit{finitary symmetric group of X}, $\Sym(X)$, is the group of permutations with finite support. Bacher and de la Harpe investigate the number theoretic properties of word lengths for such groups with respect to various generating sets of transpositions.

Given a group $G$ generated by a set $S$, for $g\in G$, the \textit{word length} $\ell_{G,S}(g)$ is  the smallest non-negative integer $n$ such that $g=s_1s_2\cdots s_n$ where $s_1,s_2,\ldots,s_n\in S\cup S^{-1}$. The smallest integer $n$ such that there exists $h$ in the conjugacy class of $g$ where $\ell_{G,S}(h)=n$ is called the \textit{conjugacy length} $\kappa_{G,S}(g)$.  Let $\gamma_{G,S}(n)\in\N\cup\{0\}\cup\{\infty\}$ denote the number of conjugacy classes in $G$ made up of elements $g$ where $\kappa_{G,S}(g)=n$ for $n\in\N$.

If the pair $(G,S)$ is such that $\gamma_{G,S}(n)$ is finite for all $n\in\N$, then Bacher and de la Harpe define the \textit{conjugacy growth series} as
\begin{equation}
C_{G,S}(q)=\sum_{n=0}^{\infty}\gamma_{G,S}(n)q^n.
\end{equation}
They also define the \textit{exponential rate of conjugacy growth} to be 
\begin{equation*}
H_{G,S}^{\text{conj}}=\limsup_{n\rightarrow\infty}\frac{\log\gamma_{G,S}(n)}{n}.
\end{equation*}
In the cases we study, the values of $H_{G,S}^{\text{conj}}$ are $0$; thus, we define the \textit{modified exponential rate of conjugacy growth} to be
\begin{equation}
\widetilde{H}_{G,S}^{\text{conj}}=\limsup_{n\rightarrow\infty}\frac{\log\gamma_{G,S}(n)}{\rt{n}}.
\end{equation}

Extending classical facts about finite symmetric groups, a natural bijection can be seen between the conjugacy classes of $\Sym(X)$ and sets of integer partitions. Motivated by their study of subgroups of $\Sym(X)$, Bacher and de la Harpe define \textit{generalized partition functions}.  Given a vector $\textbf{e}:=(e_1,e_2,\ldots,e_k)\in\Z^{k}$, the corresponding \textit{generalized partition function} $p(n)_{\textbf{e}}$ is defined as the coefficients of the power series
\begin{align}
\ds\sum_{n=0}^{\infty}p(n)_{\textbf{e}}q^n&=\prod_{m=1}^{k}\prod_{n=1}^{\infty}\frac{1}{(1-q^{mn})^{e_m}}=\prod_{n=1}^{\infty}\frac{1}{(1-q^{n})^{e_1}\cdots(1-q^{kn})^{e_k}}.
\end{align}
The function $p(n)_{\textbf{e}}$ can be interpreted as multi-partition numbers with constraints on the parts.

The group theory in~\cite{bacher} motivates the study of the asymptotics of these power series, and the classical work of Ramanujan motivates the study of their congruences.  Here, we briefly recall the classical theory of the partition function $p(n)$.

A \textit{partition} of a positive integer $n$ is a non-increasing sequence $\lambda:=(\lambda_1,\lambda_2,\ldots)$ such that $\sum_{j\geq 1}\lambda_j=n$.  The partition function $p(n)$ counts the number of partitions of $n$.  This function has been an important object of study both for its uses in number theory and combinatorics and in its own right.  The partition function has generating function
\begin{equation}\label{p}
\sum_{n=0}^{\infty}p(n)q^n=\prod_{n=1}^{\infty}\frac{1}{1-q^{n}}.
\end{equation}
This is the case of the generalized partition function with the vector $\textbf{e}=(1)$.  By Proposition 1 in~\cite{bacher}, this series~(\ref{p}) corresponds to $C_{\Sym(\N),S}(q)$ where $S\subset\Sym(\N)$ is a generating set such that $S_{\N}^{\text{Cox}}\subset S\subset T_{\N}$, where
\begin{equation*}
S_{\N}^{\text{Cox}}=\{(i,i+1):i\in\N\}
\end{equation*}  
is such that $(\Sym(\N), S_{\N}^{\text{Cox}})$ is a Coxeter system, and 
\begin{equation*}
T_{\N}=\{(x,y)\in\Sym(\N): x,y\in\N\text{ are distinct}\}
\end{equation*} 
is the conjugacy class of all transpositions in $\Sym(\N)$.  Therefore, the famous Hardy-Ramanujan asymptotic formula
\begin{equation}\label{asymptotic}
p(n)\sim\frac{e^{\pi\rt{2n/3}}}{4n\rt{3}}
\end{equation}
as $n\rightarrow\infty$ implies that the coefficients of the conjugacy growth series defined by the set $S$ above, $\gamma_{\Sym(\N),S}(n)$, approach the right-hand side of~(\ref{asymptotic}) as $n\rightarrow\infty$.  In particular, we then have that the modified exponential rate of conjugacy growth is given by
\begin{equation*}
\widetilde{H}_{\Sym(\N),S}=\pi\rt{\frac{2}{3}}.
\end{equation*}

Using Ingham's Tauberian Theorem~\cites{bringmann,ingham}, we derive an asymptotic formula for the generalized partition function $p(n)_\textbf{e}$ for any vector $\textbf{e}$ with nonnegative integer entries.  Given $\textbf{e}:=(e_1,e_2,\ldots,e_k)$, let $d:=\gcd\{m:e_m\neq 0\}$.  Note that $p(n)_{\textbf{e}}=0$ for all $n\geq0$ such that $d\nmid n$.  Define quantities $\beta$, $\gamma$, and $\delta$ by
\begin{equation}\label{beta}
\beta:=\beta(\textbf{e})=\sum_{n=1}^{k/d}ne_{dn},
\end{equation}
\begin{equation}\label{gamma}
\gamma:=\gamma(\textbf{e})=\sum_{n=1}^{k/d}e_{dn},
\end{equation}
and
\begin{equation}\label{delta}
\delta:=\delta(\textbf{e})=\sum_{n=1}^{k/d}\frac{e_{dn}}{n}.
\end{equation}
In terms of these constants, we obtain the following asymptotics.

\begin{thm}\label{as}
Assume the notation above.  Given a nonzero vector $\textbf{e}:=(e_1,e_2,\ldots,e_k)\in\Z^k$ where $e_m\geq 0$ for all $m$, as $n\rightarrow\infty$, we have that
\begin{equation*}
p(dn)_{\textbf{e}}\sim\frac{\lambda A^{\frac{1+\gamma}{4}}}{2\rt{\pi}n^{\frac{3+\gamma}{4}}}e^{2\rt{An}},
\end{equation*}
where
\begin{equation*}
\lambda:=\prod_{m=1}^{k}\left(\frac{m}{2\pi}\right)^{\frac{e_{dm}}{2}}
\end{equation*}
and
\begin{equation*}
A:=\frac{\pi^2\delta}{6}.
\end{equation*}
\end{thm}

\begin{remark}
Using the circle method~\cite{apostol}, one can obtain stronger forms of Theorem~\ref{as} with explicit error terms.
\end{remark}

\begin{example}
Let $\textbf{e}=(1)$.  Then $d=1$, $\gamma=1$, and $\delta=1$, so $\lambda=\frac{1}{\rt{2\pi}}$ and $A=\frac{\pi^2}{6}$.  Then as $n\rightarrow\infty$, we have that
\begin{equation*}
p(n)_{(1)}\sim\frac{e^{\pi\sqrt{2 n/3}}}{4 n \sqrt{3}},
\end{equation*}
and our asymptotic coincides with~(\ref{asymptotic}).
\end{example}

\begin{example}
Let $\textbf{e}=(2)$.  Then $d=1$, $\gamma=2$, and $\delta=2$, so $\lambda=\frac{1}{2\pi}$ and $A=\frac{\pi^2}{3}$.  Then as $n\rightarrow\infty$, we have that
\begin{equation*}
p(n)_{(2)}\sim\frac{e^{2\pi\sqrt{n/3}}}{4\cdot 3^{\frac{3}{4}}n^{\frac{5}{4}}}.
\end{equation*}
Similarly, for $\textbf{e}=(0,1)$, as $n\rightarrow\infty$, we have that
\begin{equation*}
p(2n)_{(0,1)}\sim\frac{e^{\pi\rt{2 n/3}}}{4n\rt{3}}.
\end{equation*}
By Proposition 11 in~\cite{bacher}, it is known that
\begin{equation}
C_{\Alt(\N),S'}(q)=\frac{1}{2}\sum_{n=0}^{\infty}p(n)_{(2)}q^n+\frac{1}{2}\sum_{n=0}^{\infty}p(n)_{(0,1)}q^n,
\end{equation}
where $S'\subset\Alt(\N)$ is a generating set such that $S_{\N}^{A}\subset S'\subset T_{\N}^{A}$.  Here we have that
\begin{equation*}
S_{\N}^A:=\{(i, i+1, i+2)\in\Alt(\N) : i\in N\}
\end{equation*}
and 
\begin{equation*}
T_{\N}^A:=\cup_{g\in\Alt(\N)}g S_{\N}^A g^{-1}.
\end{equation*}
Then the coefficients of this series satisfy the asymptotic
\begin{equation*}
\gamma_{\Alt(\N),S'}(n)\sim \frac{e^{2\pi\rt{n/3}}}{3^{\frac{3}{4}}\cdot 8n^{\frac{5}{4}}} 
\end{equation*}
as $n\rightarrow\infty$.  Therefore, we have that 
\begin{equation*}
\widetilde{H}_{\Alt(\N),S'}=\frac{2\pi}{\rt{3}}.
\end{equation*}
\end{example}

In addition to finding generalized asymptotic formulas, we study generalized forms of Ramanujan's congruences including those conjectured by Bacher and de la Harpe in~\cite{bacher}.  The Ramanujan congruences are~\cite{berndt}:
\begin{align*}
p(5n+4)&\equiv 0\pmod 5\\
p(7n+5)&\equiv 0\pmod 7\\
p(11n+6)&\equiv 0\pmod {11}.
\end{align*}

Using the definition of generalized partition numbers, Bacher and de la Harpe define a \textit{generalized Ramanujan congruence} as:
\begin{enumerate} [label=(\roman*)]
\item a nonzero integer vector $\textbf{e}:=(e_1,e_2,\ldots,e_k)\in\Z^k$,
\item an arithmetic progression $(An+B)_{n\geq 0}$ with $A\geq 2$ and $1\leq B\leq A-1$, and
\item a prime power $\ell^f$ with $\ell$ prime and $f\geq 1$
\end{enumerate}
such that
\begin{equation*}
p(An+B)_{\textbf{e}}\equiv 0\pmod{\ell^f}
\end{equation*}
for all $n\geq 0$.  

\begin{remark}
In Theorem~\ref{as} the entries of the vector $\textbf{e}$ must be nonnegative, whereas here the entries of the vector $\textbf{e}$ are allowed to take on negative values.
\end{remark}

Bacher and de la Harpe conjecture 284 generalized Ramanujan congruences for $p(n)_{\textbf{e}}$.  They observe how the coefficients of conjugacy growth series satisfy congruence relations similar to the classic Ramanujan congruences for the partition function, and use these congruences to analyze the finitary alternating group.

There are two different types of congruences of the form $p(\ell n+B)_{\textbf{e}}\equiv 0\pmod\ell$ that appear in~\cite{bacher}.  In the first type, the value of $B$ is uniquely determined by the vector $\textbf{e}$. The second type consists of sets of congruences of the form $p(\ell n+B)_{\textbf{e}}\equiv 0\pmod\ell$ with varying values of $B$ using the same values of $\ell$ and $\textbf{e}$.  
One example of the first type is the conjectured congruence 
\begin{align*} 
&p(5n+2)_{(2,0,0,4)}\equiv 0\pmod 5
\end{align*}
for all $n\geq 0$.  An example of a set of the second type is the pair of conjectured congruences
\begin{align*}
&p(5n+2)_{(2,0,0,2)}\equiv p(5n+3)_{(2,0,0,2)}\equiv 0\pmod 5
\end{align*}
for all $n\geq 0$.

We offer an algorithm to determine the number of values of $p(n)_{\textbf{e}}$ that must be computed in order to guarantee a congruence.  Given a vector $\textbf{e}:=(e_1,e_2,\ldots,e_k)\in\Z^k$ and a prime $\ell\geq 5$, we construct a vector of nonnegative integers $\textbf{c}_{\textbf{e}}:=(c_1,c_2,\ldots,c_k)$.  Let $\textbf{e}':=\textbf{e}-\ell\textbf{c}_{\textbf{e}}$.  
We define
\begin{equation}\label{alpha}
\alpha:=\sum_{m=1}^{k}me_m
\end{equation}
and
\begin{equation}
\delta_{\ell}:=\begin{cases}
\frac{\alpha}{24}\pmod\ell & \ell\nmid 24\\
0 & \ell\mid 24
\end{cases}
\end{equation}where $\frac{1}{24}$ is taken as the multiplicative inverse of 24 $\pmod\ell$.
We then define 
\begin{equation}\label{w}
w:=-\frac{1}{2}\sum_{m=1}^{k}e_m',
\end{equation}
\begin{equation}\label{omega}
\omega:=\frac{1}{24}\sum_{m=1}^{k}me_m',
\end{equation}
and 
\begin{equation}\label{N}
N:=24 N_0\gcd(24,\sum_{m=1}^{k}\frac{N_0}{m}e_m')^{-1},
\end{equation}
where $N_0:=\lcm\{m : e_m'\neq 0\}$. The vector $\textbf{e}'$ that we construct satisfies the following conditions:
\begin{enumerate}[label=(\roman*)]
\item $e_m'\leq 0$ for all $m$,
\item $\omega\in\Z$,
\item $w\in\Z$, and 
\item $\sum_{m=1}^{k}\frac{N}{m}e_m'\equiv 0\pmod{24}$.
\end{enumerate}
We then define
\begin{equation}\label{K}
K_{\textbf{e}}:=\left\lfloor{\frac{w}{12}N\prod_{p\mid N}\left(1+\frac{1}{p}\right)}\right\rfloor+\frac{\omega-\delta_{\ell}}{\ell},
\end{equation}
where the product runs over all prime divisors of $N$.

Using this notation, we arrive at the following theorem:
\begin{thm}\label{type1}
Assume the notation above.  Let $\ell\geq 5$ be prime.  Then $p(\ell n+\delta_{\ell})_{\textbf{e}}\equiv 0\pmod\ell$ for all $n$ if and only if $p(\ell n+\delta_{\ell})_{\textbf{e}}\equiv 0\pmod\ell$ for all $0\leq n\leq K_{\textbf{e}}$.
\end{thm}

The second type of congruence conjectured by~\cite{bacher} relies on much of the same notation and machinery, but requires considering the Legendre symbol with respect to the prime $\ell$.  We define two sets as follows:
\begin{equation}\label{S1}
S_{+}:=\left\{\gamma_{\ell}\in\Z : \left(\frac{\gamma_{\ell}-\delta_{\ell}}{\ell}\right)=1\text{ and }0\leq\gamma_{\ell}\leq\ell-1\right\}
\end{equation}
and
\begin{equation}\label{S-1}
S_{-}:=\left\{\gamma_{\ell}\in\Z : \left(\frac{\gamma_{\ell}-\delta_{\ell}}{\ell}\right)=-1\text{ and }0\leq\gamma_{\ell}\leq\ell-1\right\}.
\end{equation}
We then define
\begin{equation}\label{K'}
K_{\textbf{e}}':=\left\lfloor{\frac{w}{12}N\ell^2\prod_{p\mid N\ell^2}\left(1+\frac{1}{p}\right)}\right\rfloor+\frac{\omega-\delta_{\ell}}{\ell},
\end{equation}
where the product runs over all prime divisors of $N\ell^2$.
\begin{thm}\label{type2}
Assume the notation above.  Let $\ell\geq 2$ be prime where if $\ell=2$ or $3$, $\alpha\equiv 0\pmod{\ell}$.  Then $p(\ell n+\gamma_{\ell})_{\textbf{e}}\equiv 0\pmod{\ell}$ for all $n$ and all $\gamma_{\ell}\in S_{+}$ (resp. $S_{-}$) if and only if $p(\ell n+\gamma_{\ell})_{\textbf{e}}\equiv 0\pmod{\ell}$ for all $0\leq n\leq K_{\textbf{e}}'$ and all $\gamma_{\ell}\in S_{+}$ (resp. $S_{-}$).
\end{thm}
Using Theorems~\ref{type1} and~\ref{type2}, we arrive at the following corollary:
\begin{cor}\label{cor}
All of the conjectured congruences in~\cite{bacher} are true.
\end{cor}

\begin{remark}
Theorem~\ref{type1} and~\ref{type2} can be generalized to congruences modulo prime powers $\ell^f$ in a straightforward way.
\end{remark}

The results in this paper are obtained by making use of the theory of modular forms.  In Section 2.1 of this paper, we cover preliminaries on modular forms.  Section 2.2 focuses on Ingham's Tauberian Theorem, which we will use to prove Theorem~\ref{as}.  Section 2.3 covers Sturm's Theorem and additional properties used to prove Theorems~\ref{type1} and~\ref{type2}. In Section 3 we prove Theorem~\ref{as} and provide an example of an asymptotic.  Following this, in Section 4, we give an algorithm used to construct a vector $\textbf{c}_e$ that we use in the proof of Theorem~\ref{type1}. Section 5 of the paper is dedicated to the proofs of Theorems~\ref{type1} and~\ref{type2}, and Section 6 looks at an example of each type of congruence. Section 7 is an Appendix listing all congruences conjectured by Bacher and de la Harpe in ~\cite{bacher}, and proved using Theorems~\ref{type1} and~\ref{type2}.

\section*{Acknowledgments}
The authors would like to thank Ken Ono and Olivia Beckwith for advising this project and for their many helpful conversations and suggestions. The authors would also like to thank Pierre de la Harpe and Roland Bacher for their comments and suggestions on a previous version of this paper. Along with this, the authors would like to thank Emory University and the NSF for their support via grant DMS-1250467.

\section{Preliminaries on Modular Forms}
\subsection{Modularity}
Proving Theorems~\ref{type1} and~\ref{type2} requires the use of modular forms and their properties.  Here we state standard properties of modular forms that can be found in many texts such as~\cite{apostol} and~\cite{ono}. We use the following definition of modular forms from~\cite[p. 114]{apostol}:
\begin{definition}
A function $f$ is said to be an \textit{entire modular form of weight k} on a subgroup $\Gamma\subseteq\SL_{2}(\Z)$ if it satisfies the following conditions: 
\begin{enumerate}[label=(\roman*)]
\item $f$ is analytic in the upper-half $\mathbb{H}$ of the complex plane,
\item $f$ satisfies the equation 
\begin{equation*}
f\left(\frac{az+b}{cz+d}\right)=(cz+d)^kf(z)
\end{equation*} 
whenever $\begin{pmatrix}
a& b\\
c& d
\end{pmatrix}\in\Gamma$ and $z\in\mathbb{H}$, and
\item the Fourier expansion of $f$ has the form 
\begin{equation*}
f(z)= \sum_{n=0}^{\infty} c(n)e^{2\pi i n z}
\end{equation*}
at the cusp $i\infty$, and $f$ has analogous Fourier expansions at all other cusps.
\end{enumerate}
\end{definition} 
Note that a \textit{cusp} of $\Gamma$ is an equivalence class in $\mathbb{P}^1(\mathbb{Q})=\mathbb{Q}\cup {\infty}$ under the action of $\Gamma$ ~\cite[p. 2]{ono}.

We use \textit{Dedekind's eta-function}, a weight 1/2 modular form defined as the infinite product
\begin{equation*}
\eta(z):=q^{1/24}\ds\prod_{n=1}^{\infty}(1-q^n)
\end{equation*}
where $q:=e^{2\pi iz}$ and $z\in\mathbb{H}$. 
The eta-function has the following transformation property as described in~\cite[p. 17]{ono}:
\begin{equation}\label{eta}
\eta \left(- \frac{1}{z}\right) = \left(- i z \right)^{\frac{1}{2}} \eta \left(z\right).
\end{equation}
An \textit{eta-quotient} is a function $f(z)$ of the form
\begin{equation*}
f(z):=\ds\prod_{\delta\mid N}\eta(\delta z)^{r_{\delta}},
\end{equation*}
where $N\geq 1$ and each $r_{\delta}$ is an integer.  If each $r_{\delta}\geq 0$, then $f(z)$ is known as an \textit{eta-product}.
If $N$ is a positive integer, then we define $\Gamma_0(N)$ as the congruence subgroup
\begin{equation*}
\Gamma_0(N):=\left\{\begin{pmatrix}
a& b\\
c& d
\end{pmatrix}\in SL_2(\Z):c\equiv 0\pmod N\right\}.
\end{equation*}
We will need the following fact about congruence subgroups from~\cite[p. 2]{ono}:

\begin{prop}\label{dim}
If $N$ is a positive integer, then
\begin{equation*}
[\Gamma_0(1):\Gamma_0(N)]=N\ds\prod_{p\mid N}\left(1+\frac{1}{p}\right)
\end{equation*}
where the products are over the prime divisors of $N$.
\end{prop}

We recall Theorem 1.64 from~\cite[p. 18]{ono} regarding the modularity of eta-quotients:
\begin{thm}\label{mod}
If $f(z)=\prod_{\delta\mid N}\eta(\delta z)^{r_{\delta}}$ has integer weight $k=\frac{1}{2}\sum_{\delta\mid N}r_{\delta}$, with the additional properties that
\begin{equation*}
\sum_{\delta\mid N}\delta r_{\delta}\equiv 0\pmod{24}
\end{equation*}
and
\begin{equation*}
\sum_{\delta\mid N}\frac{N}{\delta}r_{\delta}\equiv 0\pmod{24},
\end{equation*}
then $f(z)$ satisfies
\begin{equation}\label{func}
f\left(\frac{az+b}{cz+d}\right)=\chi(d)(cz+d)^kf(z)
\end{equation}
for every $\begin{pmatrix}
a& b\\
c& d
\end{pmatrix}\in\Gamma_0(N)$ where the character $\chi$ is defined by $\chi(d):=\left(\frac{(-1)^ks}{d}\right)$, where $s:=\prod_{\delta\mid N}\delta^{r_\delta}$.
\end{thm}
 
Any modular form that is holomorphic at all cusps of $\Gamma_0(N)$ and satisfies~(\ref{func}) is said to have Nebentypus character $\chi$, and the space of such forms is denoted $M_k(\Gamma_0(N),\chi)$. In particular, 
if $k$ is a positive integer and $f(z)$ is holomorphic at all of the cusps of $\Gamma_0(N)$, then $f(z)\in M_k(\Gamma_0(N),\chi)$.  Furthermore, all modular forms can be identified by their Fourier expansion.

If $f(z)$ is a modular form, then we can act on it with \textit{Hecke operators}.  If $f(z)=\sum_{n=0}^{\infty}a(n)q^n\in M_k(\Gamma_0(N),\chi)$, then the action of the \textit{Hecke operator} $T_{p,k,\chi}$ on $f(z)$ is defined by
\begin{equation*}
f(z)\mid T_{p,k,\chi}:=\ds\sum_{n=0}^{\infty}(a(pn)+\chi(p)p^{k-1}a(n/p))q^n
\end{equation*}
where $a(n/p)=0$ if $p\nmid n$.  We recall the following result from~\cite[p. 21]{ono}:
\begin{prop}\label{Hecke}
Suppose that 
\begin{equation*}
f(z)=\ds\sum_{n=0}^{\infty}a(n)q^n\in M_k(\Gamma_0(N),\chi).
\end{equation*}
If $m\geq 2$, then $f(z)\mid T_{m,k,\chi}\in M_k(\Gamma_0(N),\chi)$.
\end{prop}
We now recall the notion of a ``twist" of a modular form.  Suppose that $f(z)=\sum_{n=0}^{\infty}a(n)q^n\in M_k(\Gamma_0(N),\chi)$.  If $\psi$ is a Dirichlet character (mod $m$), then the $\psi$-\textit{twist} of $f(z)$ is defined by
\begin{equation*}
f_{\psi}(z):=\ds\sum_{n=0}^{\infty}\psi(n)a(n)q^n.
\end{equation*}
Recall that $\psi(n)=0$ if $\gcd(n,m)\neq1$.
We will use a property of ``twists'' from~\cite[p. 23]{ono}:
\begin{prop}\label{twist}
Suppose that $f(z)=\sum_{n=0}^{\infty}a(n)q^n\in M_k(\Gamma_0(N),\chi)$.  If $\psi$ is a Dirichlet character with modulus $m$, then
\begin{equation*}
f_{\psi}(z)\in M_k(\Gamma_0(Nm^2),\chi\psi^2).
\end{equation*}
\end{prop}
\begin{remark}
If $\psi$ is the Legendre symbol, then $\psi^2$ is the trivial character so $\chi\psi^2=\chi$ and $M_k(\Gamma_0(Nm^2),\chi\psi^2)= M_k(\Gamma_0(Nm^2),\chi)$.
\end{remark}

\subsection{Ingham's Tauberian Theorem}
We now look at the tool used to derive the asymptotic formula for the generalized partition function $p(n)_\textbf{e}$ for any vector $\textbf{e}$. Recall Ingham's Tauberian Theorem from~\cites{bringmann,ingham}:
\begin{thm}\label{ingham}
Let $f(q)=\sum_{n=0}^{\infty}a(n)q^n$ be a power series with weakly increasing coefficients and radius of convergence equal to 1. If there are constants $A>0$, $\lambda,\alpha \in \mathbb{R}$ such that
\begin{equation*}
f(e^{-\epsilon})\sim \lambda \epsilon ^{\alpha}e^{A/\epsilon}
\end{equation*}
as $\epsilon \rightarrow 0^+$, then as $n \rightarrow \infty$, we have
\begin{equation*}
a(n) \sim \frac{\lambda A^{\frac{\alpha}{2} + \frac{1}{4}}}{2 \rt{\pi}n^{\frac{\alpha}{2}+\frac{3}{4}}}e^{2\rt{An}}.
\end{equation*}
\end{thm}

\subsection{Sturm's Theorem}
We now introduce the machinery used to determine the number of coefficients needed to guarantee a generalized Ramanujan congruence.  Suppose that
\begin{equation*}
f=\ds\sum_{n\geq n_0}a(n)q^n
\end{equation*}
is a formal power series with coefficients in $\mathcal{O}_K$, the ring of integers of a number field $K$.  If $\mathfrak{m}\subset\mathcal{O}_K$ is an ideal, then we define $\ord_{\mathfrak{m}}(f)$, the \textit{order of $f$ modulo $\mathfrak{m}$}, by
\begin{equation}\label{ord}
\ord_{\mathfrak{m}}(f):=\min\{n:a(n)\not\in\mathfrak{m}\}.
\end{equation}
If $a(n)\in\mathfrak{m}$ for all $n$, then we let $\ord_{\mathfrak{m}}(f):=+\infty$.

Using this notation, we recall a theorem of Sturm's from~\cite[p. 40]{ono}:
\begin{thm}\label{sturm}
Let $f(z)=\sum_{n=0}^{\infty}a(n)q^n\in M_{\frac{k}{2}}(\Gamma_0(N),\chi)$ be a modular form where $k$ is a positive integer.  Furthermore, suppose that its coefficients are in $\mathcal{O}_K$, the ring of integers of a number field $K$.  If $\mathfrak{m}\subset\mathcal{O}_K$ is an ideal for which
\begin{equation*}
\ord_{\mathfrak{m}}(f)>\frac{k}{24}[\Gamma_0(1):\Gamma_0(N)],
\end{equation*}
then $\ord_{\mathfrak{m}}(f)=+\infty$.
\end{thm}

If $\mathcal{O}_K=\Z$ and $\mathfrak{m}=\la\ell\ra$, then $\ord_{\ell}(f)=\min\{n:\ell\nmid a(n)\}$ and if $\ell\mid a(n)$ for all $n$, then $\ord_{\ell}(f):=+\infty$.  Therefore Theorem~\ref{sturm} can be reformulated as follows:
\begin{cor}
Let $f(z)=\sum_{n=0}^{\infty}a(n)q^n\in M_{\frac{k}{2}}(\Gamma_0(N),\chi) \cap \mathbb{Z}[[q]]$ be a modular form where $k$ is a positive integer.  If $a(n)\equiv 0\pmod{\ell}$ for all $0\leq n\leq\frac{k}{24}[\Gamma_0(1):\Gamma_0(N)]$, then $a(n)\equiv 0\pmod{\ell}$ for all $n\geq 0$.
\end{cor}

\section{An Asymptotic Formula for $p(n)_{\textbf{e}}$}

Given $\textbf{e}:=(e_1,e_2,\ldots,e_k)$, let $d=\gcd\{m:e_m\neq 0\}$.  Define $\beta$, $\gamma$, and $\delta$ by~(\ref{beta}),~(\ref{gamma}), and~(\ref{delta}), respectively.  Define $\textbf{e}':=(e_1',e_2',\ldots,e_k')$ by $e'_{m}=e_{dm}$.

\begin{lemma}\label{d}
Assume the notation above.  Then $p(dn)_{\textbf{e}}=p(n)_{\textbf{e}'}$ for all $n\geq 0$.
\end{lemma}
\begin{proof}
This follows from a simple change of variables $q\rightarrow q^d$.
\end{proof}

\begin{proof}[Proof of Theorem~\ref{as}]
By Lemma~\ref{d}, since $p(dn)_{\textbf{e}}=p(n)_{\textbf{e}'}$ for all $n\geq 0$, it suffices to find an asymptotic for $p(n)_{\textbf{e}'}$.  First note that $\gcd\{m:e_m'\neq 0\}=1$ by definition of $\textbf{e}'$.  Now let
\begin{equation*}
f(q)=\sum_{n=0}^{\infty}p(n)_{\textbf{e}'}q^n=q^{\frac{\beta}{24}}\prod_{m=1}^{k}\eta(mz)^{-e_m'}.
\end{equation*}
Then we have
\begin{align*}
f(e^{-\epsilon})&=e^{-\frac{\beta\epsilon}{24}}\prod_{m=1}^{k}\eta\left(\frac{-m\epsilon}{2\pi i}\right)^{-e_m'}.
\end{align*}
By~(\ref{eta}), it follows that
\begin{align*}
\prod_{m=1}^{k}\eta\left(\frac{-m\epsilon}{2\pi i}\right)^{e_m'}&=\prod_{m=1}^{k}\left(\frac{2\pi}{m\epsilon}\right)^{\frac{e_m'}{2}}\eta\left(\frac{2\pi i}{m\epsilon}\right)^{e_m'}\\
&=\epsilon^{-\frac{\gamma}{2}}\prod_{m=1}^{k}\left(\frac{2\pi}{m}\right)^{\frac{e_m'}{2}}\eta\left(\frac{2\pi i}{m\epsilon}\right)^{e_m'}.
\end{align*}
Therefore, we have that
\begin{align*}
f(e^{-\epsilon})&=e^{-\frac{\beta\epsilon}{24}}\epsilon^{\frac{\gamma}{2}}\prod_{m=1}^{k}\left(\frac{m}{2\pi}\right)^{\frac{e_m'}{2}}\eta\left(\frac{2\pi i}{m\epsilon}\right)^{-e_m'}.
\end{align*}
As $\epsilon\rightarrow0^+$, it follows that
\begin{align*}
\prod_{m=1}^{k}\eta\left(\frac{2\pi i}{m\epsilon}\right)^{-e_m'}&\sim\prod_{m=1}^{k}e^{\frac{\pi^2 e_m'}{6m\epsilon}}\sim e^{\frac{\pi^2\delta}{6\epsilon}}.
\end{align*}
Then as $\epsilon\rightarrow 0^+$, we obtain
\begin{align*}
f(e^{-\epsilon})&\sim\epsilon^{\frac{\gamma}{2}}e^{\frac{\pi^2\delta}{6\epsilon}}\prod_{m=1}^{k}\left(\frac{m}{2\pi}\right)^{\frac{e_m'}{2}}\sim\lambda\epsilon^{\frac{\gamma}{2}}e^{\frac{A}{\epsilon}},
\end{align*}
where $\lambda$ and $A$ are defined in the statement of Theorem~\ref{as}.  Note that $p(n)_{\textbf{e}'}$ is supported for all $n\geq\max\{m:e_m'\neq 0\}$ since $\gcd\{m:e_m'\neq 0\}=1$, thus for all $n\geq\lcm\{m:e_m'\neq 0\}$, $p(n)_{\textbf{e}'}$ is positive.  Additionally, since each $p(n)_{\textbf{e}'}$ is a product of positive powers of the generating function for $p(n)$ with allowed changes of variable and $p(n)$ is increasing, it follows that the values of $p(n)_{\textbf{e}'}$ are weakly increasing on progressions that support the nonvanishing coefficients.  Since $p(n)_{\textbf{e}'}$ is eventually nonvanishing for all $n$, it is therefore eventually weakly increasing.

Furthermore, $f(q)$ has radius of convergence 1.  Every modular form maps the upper half plane $\mathbb{H}$ to the unit disk and thus has radius of convergence at least 1.  Since $f(q)$ has a pole at $q=1$, the radius of convergence of $f(q)$ must equal 1.
By Theorem~\ref{ingham}, it then follows that as $n\rightarrow\infty$, we have that
\begin{equation*}
p(n)_{\textbf{e}'}\sim\frac{\lambda A^{\frac{1+\gamma}{4}}}{2\rt{\pi}n^{\frac{3+\gamma}{4}}}e^{2\rt{An}}.
\end{equation*}
By Lemma~\ref{d}, as $n\rightarrow\infty$, we have that 
\begin{equation*}
p(dn)_{\textbf{e}}\sim\frac{\lambda A^{\frac{1+\gamma}{4}}}{2\rt{\pi}n^{\frac{3+\gamma}{4}}}e^{2\rt{An}}.
\end{equation*}
We thus obtain an asymptotic for $p(dn)_{\textbf{e}}$.
\end{proof}

\begin{example}
Let $\textbf{e}:=(1,0,1)$.  Then $d=1$, $\gamma=2$, and $\delta=\frac{4}{3}$, so $A=\frac{2\pi^2}{9}$ and $\lambda=\frac{\rt{3}}{2\pi}$.  Then by Theorem~\ref{as}, we have that
\begin{equation*}
p(n)_{\textbf{e}}\sim P(n)_{\textbf{e}},
\end{equation*}
where 
\begin{equation*}
P(n)_{\textbf{e}}:=\frac{1}{6\cdot2^{\frac{1}{4}}n^{\frac{5}{4}}}e^{\frac{2\pi\rt{2n}}{3}}.
\end{equation*} 
Below we display the first 10000
values of $p(n)_{\textbf{e}}$ and $P(n)_{\textbf{e}}$ (computed in Mathematica).  As $n\rightarrow\infty$, we observe that the ratio $p(n)_{\textbf{e}}/P(n)_{\textbf{e}}$ approaches 1.
\begin{table}[h!]
\caption{Ratio of $p(n)_{\textbf{e}}$ and $P(n)_{\textbf{e}}$}
\begin{center}
\begin{tabular}{ |c|c|c|c| }
\hline
$n$ & $p(n)_{\textbf{e}}$ & $P(n)_{\textbf{e}}$ & $p(n)_{\textbf{e}}/P(n)_{\textbf{e}}$\\
\hline
1000 & $1.155\cdot 10^{36}$ & $1.187\cdot 10^{36}$ & 0.97266\\
2000 & $3.459\cdot 10^{52}$ & $3.527\cdot 10^{52}$ & 0.98057\\
3000 & $1.775\cdot 10^{65}$ & $1.804\cdot 10^{65}$ & 0.98410\\
4000 & $9.855\cdot 10^{75}$ & $9,993\cdot 10^{75}$ & 0.98621\\
5000 & $2.992\cdot 10^{85}$ & $3.029\cdot 10^{85}$ &  0.98765\\
6000 & $1.145\cdot 10^{94}$ & $1.158\cdot 10^{94}$ & 0.98872\\
7000 & $9.106\cdot 10^{101}$ & $9.198\cdot 10^{101}$ & 0.98955\\
8000 & $2.079\cdot 10^{109}$ & $2.099\cdot 10^{109}$ & 0.99022\\
9000 & $1.711\cdot 10^{116}$ & $1.727\cdot 10^{116}$ & 0.99078\\
10000 & $5.990\cdot 10^{122}$ & $6.042\cdot 10^{122}$ & 0.99125\\
\hline
\end{tabular}
\end{center}
\end{table}
\end{example}
\bigskip

\section{An Algorithm for the Vector $\textbf{c}_e$}

We now establish an algorithm used to confirm or refute alleged generalized Ramanujan congruences.  Define $\alpha$ by~(\ref{alpha}).  Given a prime $\ell\geq 2$ where if $\ell=2$ or 3, $\alpha\equiv 0\pmod\ell$, and a vector $\textbf{e}:=(e_1,e_2,\ldots,e_k)\in\Z^k$ with $-\ell+1\leq e_m\leq 0$, we must construct a vector $\textbf{c}_e$ so that $\textbf{e}'=\textbf{e}-\ell\textbf{c}_e$ satisfies the following conditions:
\begin{enumerate}[label=(\roman*)]
\item $e_m'\leq 0$ for all $m$,
\item $\omega\in\Z$,
\item $w\in\Z$, and 
\item $\sum_{m=1}^{k}\frac{N}{m}e_m'\equiv 0\pmod{24}$
\end{enumerate}
where $w$, $\omega$, and $N$ are defined by~(\ref{w}),~(\ref{omega}), and~(\ref{N}), respectively.

\begin{prop}\label{cond}
Assume the notation above.  Given a prime $\ell\geq 2$ where if $\ell=2$ or $3$, $\alpha\equiv 0\pmod\ell$, and a vector $\textbf{e}:=(e_1,e_2,\ldots,e_k)\in\Z^k$ with $-\ell+1\leq e_m\leq 0$, it is possible to construct a vector $\textbf{c}_e$ such that the above conditions are satisfied.
\end{prop}

\begin{proof}
First define $\alpha$ by~(\ref{alpha}).  Then define
\begin{equation}\label{beta'}
\beta_{\textbf{e}}:=\begin{cases}
\min\{n\in\N:n\equiv\ell^{-1}\alpha\pmod{24}\text{ and }n>1\}& \ell\nmid 24\\
\min\{n\in\N:n\equiv\ell^{-1}\alpha\pmod{\frac{24}{\ell}}\text{ and }n>1\}& \ell\mid 24
\end{cases}
\end{equation}
where in the first case, $\ell^{-1}$ is taken as the multiplicative inverse of $\ell\pmod{24}$, and in the second case, since $\ell\mid\alpha$, $\ell^{-1}=\frac{1}{\ell}$.

Define $c_m'=0$ if $e_m=0$.  We now define the vector $\textbf{c}_e'$ recursively beginning with $c_k'$ as follows for all $e_m\neq 0$:
\begin{equation}\label{c_m'}
c_m'=\left\lfloor{\frac{1}{m}\left(\beta_{\textbf{e}}-\ds\sum_{n=m+1}^{k}nc_n'\right)}\right\rfloor.
\end{equation}
Note that $c_1'=\beta_{\textbf{e}}-\sum_{n=2}^{k}nc_n'$, so $\sum_{m=1}^{k}mc_m'=\beta_{\textbf{e}}$ and
\begin{align*}
\sum_{m=1}^{k}me'_m&=\sum_{m=1}^{k}me_m-\ell\ds\sum_{m=1}^{k}mc_m\\
&=\alpha-\ell\beta_{\textbf{e}}\\
&\equiv 0\pmod{24},
\end{align*}
so condition (ii) is satisfied. If $\frac{1}{2}\sum_{m=1}^{k}(e_m-\ell c_m')\in\Z$, then define $\textbf{c}_e=\textbf{c}_e'$.  

Suppose $\frac{1}{2}\sum_{m=1}^{k}(e_m-\ell c_m')\not\in\Z$.  Then choose the smallest $j$ such that $j$ is even and $c_j'>0$.  Define $c_j:=c_j'-1$ and $c_1:=c_1'+j$.  For all other $m$, let $c_m:=c_m'$.  Let $\textbf{c}_e=(c_1,c_2,\ldots,c_k)$ and define $\textbf{e}'=\textbf{e}-\ell\textbf{c}_e$.  Then $\sum_{m=1}^{k}mc_m=\sum_{m=1}^{k}mc_m'$ and $e_m'\leq 0$ for all $m$, so conditions (i)-(ii) hold.  Since $\sum_{m=1}^{k}c_m=\sum_{m=1}^{k}c_m'-1+j$ and $-1+j$ is odd, the parity of the sum $\sum_{m=1}^{k}e_m'=\sum_{m=1}^{k}e_m-\ell\sum_{m=1}^{k}c_m$ changes and $w=-\frac{1}{2}\sum_{m=1}^{k}(e_m-\ell c_m')\in\Z$.

Suppose $c_j'=0$ for all $j$ even.  Then choose the smallest $j\neq 1$ such that $j$ is odd and $c_j>0$.  Define $c_j=c_j'-1$, $c_{j-1}=c_{j-1}'+1$, and $c_1=c_1'+1$.  For all other $m$, let $c_m=c_m'$.  Let $\textbf{c}_e=(c_1,c_2,\ldots,c_k)$ and define $\textbf{e}'=\textbf{e}-\ell\textbf{c}_e$.  Then, as before, $\sum_{m=1}^{k}mc_m=\sum_{m=1}^{k}mc_m'$ and $e_m'\leq 0$ for all $m$, so conditions (i)-(ii) hold.  Since $\sum_{m=1}^{k}c_m=\sum_{m=1}^{k}c_m'+1$, the parity of the sum $\sum_{m=1}^{k}e_m'$ changes and $w\in\Z$.

Note that by the definition of $N$ in~(\ref{N}), $\sum_{m=1}^{k}\frac{N}{m}e_m'\equiv 0\pmod{24}$, so condition (iv) holds.  Thus the vector $\textbf{e}'$ satisfies conditions (i)-(iv) as desired.
\end{proof}

\section{Generalized Ramanujan Congruences}
We now use the algorithm in Proposition~\ref{cond} to prove Theorems~\ref{type1} and~\ref{type2} and establish a method of confirming or refuting alleged generalized Ramanujan congruences that fall into two different types.
First note the following fact from~\cite{bacher}:
\begin{prop}\label{e'}
Consider a vector $\textbf{e}:=(e_1,e_2,\ldots,e_k)\in\Z^k$, an arithmetic progression $(An+B)_{n\geq 0}$ with $A\geq 2$ and $1\leq B\leq A-1$, a prime $\ell$, and another vector $\textbf{e}'=(e_1',e_2',\ldots,e_k')\in\Z^k$.  Assume that $e_m'\equiv e_m\pmod{\ell}$ for all $m\geq 0$.  Then $p(An+B)_{\textbf{e}}\equiv 0\pmod{\ell}$ for all $n\geq 0$ if and only if $p(An+B)_{\textbf{e}'}\equiv 0\pmod{\ell}$ for all $n\geq 0$.  
\end{prop}

\subsection{Proof of Theorem~\ref{type1}}
By Proposition~\ref{e'}, it suffices to consider vectors $\textbf{e}=(e_1,e_2,\ldots,e_k)$ with $-\ell+1\leq e_m\leq 0$ for all $m$.  Define $\textbf{e}'=\textbf{e}-\ell\textbf{c}_e$ by Proposition~\ref{cond}.  Then since $e'_m\equiv e_m\pmod{\ell}$ for all $m\geq 0$, by Proposition~\ref{e'}, it is enough to show that $p(\ell n+\delta_{\ell})_{\textbf{e}'}\equiv 0\pmod{\ell}$ for all $n\geq 0$.  Note that
\begin{align*}
\ds\sum_{n=0}^{\infty}p(n)_{\textbf{e}'}q^n&=\ds\prod_{m=1}^{k}\ds\prod_{n=1}^{\infty}\frac{1}{(1-q^{mn})^{e'_m}}\\
&=q^{\omega}\ds\prod_{m\mid N}\eta(mz)^{-e'_m}\\
&=:q^{\omega}g(z),
\end{align*}
where $\omega$ is defined by~(\ref{omega}).  Note that $\omega\equiv\delta_{\ell}\pmod{\ell}$.

Now, note that $g(z)$ has weight $w=-\frac{1}{2}\sum_{m=1}^{k}e_m'$.  By condition (iii), $w$ must be an integer.  Furthermore, based on our choices of $\textbf{c}_e$ and $N$, $\textbf{e}'$ satisfies conditions (ii) and (iv), which are the necessary conditions of Theorem~\ref{mod}.  Since $g(z)$ is additionally holomorphic at all the cusps of $\Gamma_0(N)$, $g(z)$ is a modular form in the space $M_w(\Gamma_0(N),\chi)$.  We can therefore write its Fourier expansion
\begin{align*}
g(z):=\ds\sum_{n=0}^{\infty}b(n)q^n.
\end{align*}
Then $p(\ell n+\omega)_{\textbf{e}'}\equiv 0\pmod{\ell}$ for all $n\geq 0$ if and only if $b(\ell n)\equiv 0\pmod{\ell}$ for all $n\geq 0$.  Note that $p(\ell n+\omega)_{\textbf{e}'}\equiv 0\pmod{\ell}$ for all $n\geq 0$ if and only if $p(\ell n+\delta_{\ell})_{\textbf{e}'}\equiv 0\pmod{\ell}$ for all $n\geq \frac{\omega-\delta_{\ell}}{\ell}$ since $\omega\equiv\delta_{\ell}\pmod{\ell}$.  Therefore, we obtain that $p(\ell n+\delta_{\ell})_{\textbf{e}'}\equiv 0\pmod{\ell}$ for all $n\geq \frac{\omega-\delta_{\ell}}{\ell}$ if and only if $b(\ell n)\equiv 0\pmod{\ell}$ for all $n\geq 0$.

We now act on $g(z)$ with the Hecke operator $T_{\ell,w,\chi}$ and define
\begin{align*}
f(z)&:=g(z)\mid T_{\ell,w,\chi}\\
&=\ds\sum_{n=0}^{\infty}(b(\ell n)+\chi(\ell)\ell^{w-1}b(n/\ell))q^n.
\end{align*}
By Proposition~\ref{Hecke}, the function $f(z)$ is a modular form in the space $M_w(\Gamma_0(N),\chi)$, so we can write its Fourier series expansion as
\begin{equation*}
f(z):=\ds\sum_{n=0}^{\infty}a(n)q^n.
\end{equation*}
Then we observe that $a(n)=b(\ell n)+\chi(\ell)\ell^{w-1}b(n/\ell)$, so $a(n)\equiv b(\ell n)\pmod{\ell}$ for all $n\geq 0$.  Thus we have that $b(\ell n)\equiv 0\pmod{\ell}$ for all $n\geq 0$ if and only if $a(n)\equiv 0\pmod{\ell}$ for all $n\geq 0$.  

Since $f(z)$ has weight $w$ and is a level $N$ modular form, by Theorem~\ref{sturm}, $a(n)\equiv 0\pmod{\ell}$ for all $n\geq 0$ if and only if $a(n)\equiv 0\pmod{\ell}$ for all $0\leq n\leq\frac{w}{12}[\Gamma_0(1):\Gamma_0(N)]$.  Then we have that $p(\ell n+\delta_{\ell})_{\textbf{e}}\equiv 0\pmod{\ell}$ for all $n\geq\frac{\omega-\delta_{\ell}}{\ell}$ if and only if $p(\ell n+\delta_{\ell})_{\textbf{e}}\equiv 0\pmod{\ell}$ for all $\frac{\omega-\delta_{\ell}}{\ell}\leq n\leq \frac{\omega-\delta_{\ell}}{\ell}+\frac{w}{12}[\Gamma_0(1):\Gamma_0(N)]$.  Therefore, we obtain that $p(\ell n+\delta_{\ell})_{\textbf{e}}\equiv 0\pmod{\ell}$ for all $n\geq 0$ if and only if $p(\ell n+\delta_{\ell})_{\textbf{e}}\equiv 0\pmod{\ell}$ for all $0\leq n\leq \frac{\omega-\delta_{\ell}}{\ell}+\frac{w}{12}[\Gamma_0(1):\Gamma_0(N)]$.  By Proposition~\ref{dim}, we have that
\begin{equation*}
[\Gamma_0(1):\Gamma_0(N)]=N\prod_{p\mid N}\left(1+\frac{1}{p}\right),
\end{equation*}
so by our definition of $K_{\textbf{e}}$ in~(\ref{K}), we obtain $p(\ell n+\delta_{\ell})_{\textbf{e}}\equiv 0\pmod{\ell}$ for all $n\geq 0$ if and only if $p(\ell n+\delta_{\ell})_{\textbf{e}}\equiv 0\pmod{\ell}$ for all $0\leq n\leq K_{\textbf{e}}$.
\qed

\subsection{Proof of Theorem~\ref{type2}}

As in the proof of Theorem~\ref{type1}, by Proposition~\ref{e'}, it suffices to consider vectors $\textbf{e}=(e_1,e_2,\ldots,e_k)$ with $-\ell+1\leq e_m\leq 0$ for all $m$.  Define $\textbf{e}'$ using Proposition~\ref{cond}.  Again, let $g(z)=\prod_{m\mid N}\eta(mz)^{-e'_m}=\sum_{n=0}^{\infty}b(n)q^n$.  Then, as in the previous proof, we have that $p(\ell n+\gamma_{\ell})\equiv 0\pmod{\ell}$ for all $n\geq \frac{\omega-\delta_{\ell}}{\ell}$ if and only if $b(\ell n+\gamma_{\ell}-\omega)\equiv 0\pmod{\ell}$ for all $\geq \frac{\omega-\delta_{\ell}}{\ell}$, which holds if and only if $b(\ell n+\gamma_{\ell}-\delta_{\ell})\equiv 0\pmod{\ell}$ for all $n\geq 0$.

Define the following Dirichlet characters:
\begin{equation*}
\psi_0(n):=\begin{cases} 
      1 & \gcd(n,\ell)=1\\
      0 & \text{otherwise} 
   \end{cases}
\end{equation*}
and
\begin{equation*}
\psi_1(n):=\left(\frac{n}{\ell}\right).
\end{equation*}
Note that $\psi_0^2(n)$ and $\psi_1^2(n)$ both yield the trivial character.  Now define
\begin{equation*}
G(z):=g_{\psi_0}(z)=\ds\sum_{\ell\nmid n}b(n)q^n
\end{equation*}
and
\begin{equation*}
G_{\psi_1}(z)=\ds\sum_{\ell\nmid n}\left(\frac{n}{\ell}\right)b(n)q^n.
\end{equation*}
By Proposition~\ref{twist}, we have that $G(z)$ and $G_{\psi_1}(z)$ are both modular forms; in particular, since $\psi_0^2$ and $\psi_1^2$ are both trivial, we have that $G(z), G_{\psi_1}(z)\in M_w(\Gamma_0(N\ell^2),\chi)$.  Now define
\begin{equation*}
H_{+}(z):=\frac{G(z)+G_{\psi_1}(z)}{2}=\ds\sum_{\left(\frac{n}{\ell}\right)=1}b(n)q^n
\end{equation*}
and
\begin{equation*}
H_{-}(z):=\frac{G(z)-G_{\psi_1}(z)}{2}=\ds\sum_{\left(\frac{n}{\ell}\right)=-1}b(n)q^n.
\end{equation*}
Then we observe that $H_{\pm}(z)\in M_w(\Gamma_0(N\ell^2),\chi)$.  Recalling our definitions of the sets $S_{\pm}$ in~(\ref{S1}) and~(\ref{S-1}), note that 
\begin{equation*}
H_{\pm}(z)=\sum_{\substack{n\equiv\gamma_{\ell}+\delta_{\ell}\\(\text{mod  }\ell)\\ \gamma_{\ell}\in S_{\pm}}}b(n)q^n.
\end{equation*}
Now, we can write the Fourier expansion of $H_{\pm}(z)$ as 
\begin{equation*}
H_{\pm}(z):=\sum_{n=0}^{\infty}a_{\pm}(n)q^n.
\end{equation*}
Since $a_{+}(n)$ is only supported where $n\equiv\gamma_{\ell}+\delta_{\ell}\pmod{\ell}$ where $\gamma_{\ell}\in S_{+}$, we have that $b(\ell n+\gamma_{\ell}-\delta_{\ell})\equiv 0\pmod{\ell}$ for all $n\geq 0$ and for all $\gamma_{\ell}\in S_{+}$ if and only if $a_{+}(n)\equiv 0\pmod{\ell}$ for all $n\geq 0$.  By Theorem~\ref{sturm}, $a_{+}(n)\equiv 0\pmod{\ell}$ for all $n\geq 0$ if and only if $a_{+}(n)\equiv 0\pmod{\ell}$ for all $0\leq n\leq \frac{w}{12}[\Gamma_0(1):\Gamma_0(N\ell^2)]$.  As $a_{+}(n)\equiv 0\pmod{\ell}$ for all $0\leq n\leq \frac{w}{12}[\Gamma_0(1):\Gamma_0(N\ell^2)]$ if and only if $p(\ell n+\gamma_{\ell})_{\textbf{e}}\equiv 0\pmod{\ell}$ for all $\frac{\omega-\delta_{\ell}}{\ell}\leq n\leq \frac{\omega-\delta_{\ell}}{\ell}+\frac{w}{12}[\Gamma_0(1):\Gamma_0(N\ell^2)]$ and for all $\gamma_{\ell}\in S_{+}$, we obtain that $p(\ell n+\gamma_{\ell})_{\textbf{e}}\equiv 0\pmod{\ell}$ for all $n\geq \frac{\omega-\delta_{\ell}}{\ell}$ and for all $\gamma_{\ell}\in S_{+}$ if and only if $p(\ell n+\gamma_{\ell})_{\textbf{e}}\equiv 0\pmod{\ell}$ for all $\frac{\omega-\delta_{\ell}}{\ell}\leq n\leq \frac{\omega-\delta_{\ell}}{\ell}+\frac{w}{12}[\Gamma_0(1):\Gamma_0(N\ell^2)]$ and for all $\gamma_{\ell}\in S_{+}$.  

By our definition of $K_{\textbf{e}}'$ in~(\ref{K'}), we have that $p(\ell n+\gamma_{\ell})_{\textbf{e}}\equiv 0\pmod{\ell}$ for all $n\geq 0$ and for all $\gamma_{\ell}\in S_{+}$ if and only if $p(\ell n+\gamma_{\ell})_{\textbf{e}}\equiv 0\pmod{\ell}$ for all $0\leq n\leq K_{\textbf{e}}'$ and for all $\gamma_{\ell}\in S_{+}$.  Thus the theorem holds for $\gamma_{\ell}\in S_{+}$.  Replacing $S_{+}$ by $S_{-}$, $a_{+}(n)$ by $a_{-}(n)$, and $H_{+}(z)$ by $H_{-}(z)$, the same argument works for $\gamma_{\ell}\in S_{-}$.
\qed

\section{Examples of Congruences}

Given an alleged congruence of the form $p(\ell n+B)_\textbf{e}\equiv 0\pmod{\ell}$ that falls into either the Theorem~\ref{type1} or Theorem~\ref{type2} case, we can use the finite algorithm in Section 3 and Theorems~\ref{type1} and~\ref{type2} either to confirm or refute it.  We first use the algorithm to determine $K_{\textbf{e}}$ and $K'_{\textbf{e}}$.  By Theorems~\ref{type1} and~\ref{type2}, it suffices to check numerically that the conjectured congruences hold for all $0\leq n\leq K_{\textbf{e}}$ or $K'_{\textbf{e}}$ respectively.

\begin{example}
We have that $p(5n+2)_{(2,0,0,4)}\equiv 0\pmod{5}$ for all $n\geq 0$, as conjectured by~\cite{bacher}.
\end{example}

\begin{proof}
Note that $\alpha=18$, so $\delta_{\ell}=2$; this is an example of the Theorem~\ref{type1} case.  Using our algorithm, we have $\textbf{c}_e=(2,0,0,4)$, so $\textbf{e}'=(-8,0,0,-16)$.  Then $w=12$ and $N=4$, so $K_\textbf{e}=6$.  Computing the first 7 values of $p(5n+2)_{(2,0,0,4)}$, we find that they are equivalent to $0\pmod{5}$.  Thus the congruence holds.
\end{proof}

\begin{example}
We have that 
$p(5n+2)_{(2,0,0,2)}\equiv p(5n+3)_{(2,0,0,2)}\equiv 0\pmod{5}$ for all $n\geq 0$, as conjectured by~\cite{bacher}.
\end{example}

\begin{proof}
Note that $\alpha=10$, so $\delta_{\ell}=0$.  In this case $S_{-1}=\{2,3\}$, so this is an example of the Theorem~\ref{type2} case.  Using our algorithm, we have $\textbf{c}_e=(2,0,0,6)$ so $\textbf{e}'=(-8,0,0,-28)$.  Then $w=18$ and $N=8$, so $K_{\textbf{e}}'=540$.  Computing the first 541 values of $p(5n+2)_{(2,0,0,2)}$ and $p(5n+3)_{(2,0,0,2)}$, we find that they are equivalent to $0\pmod{5}$.  Thus the congruence holds.
\end{proof}

\section{Appendix}
We include here a list of the conjectures from~\cite{bacher}. By Corollary~\ref{cor}, they are all true. 

\subsection{Some examples of the form $p(3n+B)_{\mathbf e} \equiv 0 \pmod 3$}
\begin{flushleft}
$p(3n+2)_{(1,1),(2,1,0,2),(2,1,0,1,2,1_{10},1_{20}),(1,1,0,2,1,1_{10},2_{20})}$.
\end{flushleft}
\subsection{Some examples of the form $p(5n+B)_{\mathbf e}\equiv 0\pmod 5$}
\begin{flushleft}
$p(5n+1)_{(0,2,2),(0,4,2),(0,2,3,0,0,1)}$, \\
$p(5n+2)_{(2),(3,1),(1,3),(1,3,2),(2,0,0,2),(3,1,0,2),
(3,1,0,3),(2,0,0,4),(4,1,0,4),(1,3,4,0,0,1)}$, \\
$p(5n+2)_{(4,1,1,0,0,3),
(4,1,3,0,0,3),(3,1,1,0,0,4),(3,1,3,0,0,4),(2,2_8), (1,3,2_8),
(3,1,0,3,2_8),(4,1,0,4,2_8)}$, \\
$p(5n+3)_{(2),(4),(3,1),(1,2,0,1),(2,0,0,2),(4,0,0,2), 
(3,1,0,3),(1,2,0,3),(3,1,0,1,1_8),(2,0,0,3,1_8),(1,1,1,0,0,1)}$, \\ 
$p(5n+3)_{(1,4,3,0,0,1),(1,3,4,0,0,1),(3,3,4,0,0,1),(3,1,0,0,0,2),(2,3,4,0,0,2),(4,2,2,0,0,3),(3,2,2,0,0,4)}$, \\
$p(5n+4)_{(1),(2),(4),(2,2),(1,3),(0,2,2),(0,2,4),(1,2,0,1),(3,2,0,1),(2,1,0,3),(3,1,0,3),(3,3,0,3)}$,\\
$p(5n+4)_{(4,1,0,4),(4,3,0,4),(1,4,3,0,0,1),(3,4,3,0,0,1),(2,4,3,0,0,2),(4,1,1,0,0,3),(4,3,1,0,0,3),(1,4,3,0,0,3)}$, \\
$p(5n+4)_{(3,1,1,0,0,4),(3,3,1,0,0,4),(4,4,3_8),(1,1,0,1,3_8),(2,3,0,1,3_8),(3,4,0,4,3_8),(2,4,0,1,4_8),(3,0,0,4,4_8)}$.
\end{flushleft}   
\subsection{Some examples of the form $p(7n+B)_{\mathbf e}\equiv 0\pmod 7$}
\begin{flushleft}
$p(7n+2)_{(4),(2,2),(1,5),(3,5),(6,1,0,3),(3,5,0,3),(4,0,0,4),(1,5,0,4),(5,1,0,5),(6,1,0,6)}$,\\
$p(7n+3)_{(6),(5,1),(2,2),(1,4,0,1),(2,2,0,2),(5,1,0,4),(5,1,0,5),(2,2,0,6)}$,\\
$p(7n+4)_{(4),(6),(1,2),(2,2),(4,4),(1,5),(1,4,0,1), (3,6,0,1),(3,2,0,3),(3,5,0,3)}$,\\
$p(7n+4)_{(4,1,0,5),(5,1,0,5),(6,1,0,6),(6,5,0,6)}$,\\
$p(7n+5)_{(1),(4),(5,1),(1,5),(5,5),(2,2,0,2),(2,6,0,2),(4,3,0,3),(3,5,0,3),(3,1,0,6),(6,1,0,6),(6,3,0,6)}$,\\ 
$p(7n+6)_{(4),(6),(2,1),(5,1),(2,2),(5,3),(1,4,0,1),(4,5,0,1),(2,2,0,2),(6,2,0,2),(2,4,0,2)}$, \\
$p(7n+6)_{(3,5,0,3),(1,6,0,3),(3,3,0,4),(5,0,0,5),(5,1,0,5)}$. \\
\end{flushleft} 
\subsection{Some examples of the form $p(11n+B)_{\mathbf e}\equiv 0\pmod{11}$}
\begin{flushleft} 
$p(11n+2)_{(8),(9,1),(2,6),(1,9),(3,2,0,2),(2,6,0,2),(6,6,0,2),(3,2,0,3),(5,2,0,7),(9,1,0,9),(8,0,0,10),(10,1,0,10)}$,\\
$p(11n+3)_{(10),(4,1),(6,2),(2,6),(1,8,0,1),(5,9,0,4),(5,9,0,5),(7,2,0,7),(9,1,0,9),(6,2,0,10)}$,\\
$p(11n+4)_{(8),(2,3),(2,6),(8,9,0,1),(3,2,0,3),(3,0,0,4),(9,2,0,7),(9,1,0,9),(2,7,0,9),(9,9,0,9)}$, \\
$p(11n+5)_{(8),(6,2),(7,7),(1,9),(6,0,0,1),(3,2,0,3),(10,5,0,3),(1,2,0,4),(4,6,0,4),(5,9,0,5),(5,7,0,6),(10,1,0,10)}$, \\
$p(11n+6)_{(1),(10),(9,1),(6,2),(2,5),(2,6),(9,7),(4,2,0,1),(1,8,0,1),(2,1,0,2),(2,6,0,2),(8,7,0,3)}$, \\
$p(11n+6)_{(5,9,0,5),(3,9,0,6),(7,3,0,8),(9,0,0,9),
(9,1,0,9),(10,3,0,10)}$, \\
$p(11n+7)_{(3),(8),(9,1),(2,6),(1,9),(7,9),(4,1,0,2),(2,6,0,2),(3,2,0,3),(6,9,0,3),(4,8,0,4),(10,3,0,5), (1,0,0,6)}$, \\
$p(11n+7)_{(8,2,0,6),(5,5,0,8),(8,9,0,8),(9,1,0,9),(7,2,0,9),(3,4,0,9),(10,1,0,10)}$, \\
$p(11n+8)_{(5),(8),(10),(9,1),(6,2),(8,4),(1,9),(9,9),(1,8,0,1),(2,3,0,2),(2,6,0,2),(4,0,0,3)}$, \\
$p(11n+8)_{(3,2,0,3),(3,6,0,3),(9,1,0,4),(8,5,0,5),(5,9,0,5),(10,2,0,6)}$,\\
$p(11n+8)_{(6,4,0,6),(2,6,0,6),(1,10,0,7),(3,7,0,8),(7,1,0,10),(10,1,0,10)}$, \\
$p(11n+9)_{(7),(8),(10),(3,2),(6,2),(2,6),(6,6),(1,9),(1,8,0,1),(9,8,0,1),(2,2,0,3),(3,2,0,3),(9,4,0,3),(4,10,0,4)}$, \\
$p(11n+9)_{(5,2,0,5),(6,7,0,5),(2,9,0,5),(5,9,0,5),(10,1,0,7),(8,0,0,8),(1,9,0,8),(9,1,0,9),(10,1,0,10),(5,3,0,10)}$, \\
$p(11n+10)_{(10),(9,1),(5,2),(6,2),(1,4),(4,8),(1,8,0,1),(7,10,0,1),(2,6,0,2),(9,7,0,2)} $, \\
$p(11n+10)_{(5,9,0,2),(2,1,0,4),(7,2,0,5),(4,9,0,5),(5,9,0,5),(6,6,0,6),(8,3,0,7)}$ \\
$p(11n+10)_{(7,9,0,7),(10,0,0,8),(6,2,0,8),(4,1,0,9),(1,8,0,9),(3,5,0,10),(10,7,0,10)}$.\\
\end{flushleft} 
\subsection{Some examples of the form $p(13n+B)_{\mathbf e}\equiv 0\pmod{13}$}
\begin{flushleft} 
$p(13n+2)_{(11,1),(2,8),(2,8,0,2),(8,8,0,6),(11,1,0,11),(5,6,0,11)}$, \\
$p(13n+3)_{(12),(8,2),(1,10,0,1),(5,0,0,5),(10,6,0,6),(3,10,0,9)}$,\\
$p(13n+4)_{(10),(12),(8,2),(2,8),(1,11),(2,6,0,1),(1,10,0,1),(3,4,0,3),\dots}$,\\
$p(13n+5)_{(10),(11,1),(1,11),(6,1,0,2),(2,8,0,2),(3,4,0,3),\dots}$,\\
$p(13n+6)_{(12),(11,1),(8,2),(2,8),(1,10,0,1),(2,8,0,2),(8,12,0,2),\dots}$,\\
$p(13n+7)_{(10),(11,1),(8,2),(6,3),(1,11),(2,8,0,2),(10,10,0,2),(3,4,0,3),\dots}$,\\
$p(13n+8)_{(10),(12),(8,1),(11,1),(8,2),(1,10,0,1),(2,8,0,2),(12,8,0,2),(8,10,0,2),\dots}$,\\
$p(13n+9)_{(10),(2,8),(1,11),(10,12),(12,9,0,1),(1,6,0,2),(10,8,0,2),\dots}$,\\
$p(13n+10)_{(12),(8,2),(2,8),(12,10),(8,12),(1,7,0,1),(1,10,0,1),(5,1,0,3),\dots}$,\\
$p(13n+11)_{(10),(12),(11,1),(8,2),(1,8),(10,10),(1,11),(3,5,0,1),(2,8,0,2),\dots}$,\\
$p(13n+12)_{(10),(3,6),(2,8),(12,8),(1,11),(5,3,0,1),(1,5,0,1),(7,0,0,2),\dots}$. \\
\end{flushleft} 

\medskip

\bibliography{references}

\end{document}